\documentclass{amsart}

\usepackage[utf8]{inputenc}
\usepackage{amssymb}
\usepackage{amsfonts}
\usepackage{amscd}
\usepackage{amstext}
\usepackage{latexsym,amssymb}
\usepackage{amsmath}
\usepackage{graphicx}
\usepackage{color}
\usepackage{palatino} 
\usepackage{mathtools}
\usepackage{amsthm,mathabx}

\newcommand{\R}{\mathrm{I\!R}}
\DeclareMathOperator{\diver}{div}
\DeclareMathOperator{\dist}{dist}

\newtheorem{definition}{Definition}[section]
\newtheorem{lemma}[definition]{Lemma}
\newtheorem{proposition}[definition]{Proposition}
\newtheorem{theorem}[definition]{Theorem}

\title[Singular problems with unbounded convection term]{Positive solutions of singular elliptic problems with unbounded exponents and unbounded convection term}
 
\author{Anderson L. A. de Araujo}\thanks{First author was supported by FAPEMIG  APQ-02375-21, RED-00133-21 and CNPq.}
\address[Anderson L.A. de Araujo]{Departamento de Matem\'atica, Universidade Federal de Vi\-\c{c}osa, 36570-900, Vi\c{c}osa, MG, Brazil}
\email{anderson.araujo@ufv.br}

\author{Hamilton P. Bueno}
\address[Hamilton P. Bueno]{Departamento de Matem\'atica, Universidade Federal de Minas Ge\-rais, 31270-901 - Belo Horizonte - MG, Brazil}
\email{hamilton@mat.ufmg.br}

\author{Kamila F. L. Madalena}
\address[Kamila F. L. Madalena]{Departamento de Matem\'atica, Universidade Federal de Minas Ge\-rais, 31270-901 - Belo Horizonte - MG, Brazil}
\email{kamilalobo@ufmg.br}
\begin{document}
	
\subjclass{35J92, 35B09, 35B33, 35J25, 35J62}

\keywords{quasilinear elliptic equation, positive solution, convection growth, supercritical growth}.

\date{}

\begin{abstract} In this paper, we study the existence of a solution for a class of Dirichlet problems with a singularity and a convection term. Precisely, we consider the existence of a positive solution to the Dirichlet problem 
$$-\Delta_p u = \frac{\lambda}{u^{\alpha}} + f(x,u,\nabla u)$$
in a bounded, smooth domain $\Omega$. The convection term has exponents with no upper limitations neither in $u$ nor in $\nabla u$. This is somewhat unexpected and rare. So, we address a wide range of problems not yet contained in the literature.

The solution of the problem combines the definition of an auxiliary problem, the method of sub- and super-solution and Schauder's fixed point theorem. 
\end{abstract}

\maketitle

\section{Introduction}
In this paper, we shall establish results on the existence of positive solutions to the following $p$-Laplacian problem with a singular nonlinearity and a nonlinear convection term.
\begin{equation} \label{ps}
\left\{
\begin{array}{rcll}
-\Delta_p u & =& \frac{\lambda}{u^{\alpha}} + f(x,u,\nabla u) &\textrm{in }\ \Omega\\
u & >& 0 &\textrm{in }\ \Omega\\
u & = &0 &\textrm{on }\  \partial\Omega,
\end{array}
\right.
\end{equation}
where $-\Delta_{p} u=-\diver\left(|\nabla u|^{p-2} \nabla u\right)$ is the $p$-Laplacian operator for $1<p<\infty$, $\Omega \subset \mathbb{R}^N$ ($N\geq 2$) is a bounded, smooth domain, $0<\alpha<1$, and $f\colon \Omega \times \mathbb{R} \times \mathbb{R}^N \to \mathbb{R}$ is a continuous nonlinearity satisfying 
\begin{equation} \label{f}
    0 \leq f(x,t, \xi) \leq a|t|^{r_1}+b|\xi|^{r_2} \quad a,b>0,
\end{equation}
where $\lambda> 0$ is a parameter and $r_1,r_2 \in (0,p-1)\cup (p-1,\infty)$. 

Note that right-hand side of equation \eqref{ps} presents two main difficulties. One is the presence of the singular term, which makes it unbounded, and therefore we cannot directly use the techniques applied in Bueno and Ercole \cite{Bueno-G}, see a brief review of this paper in the sequel. Another difficulty is that the exponents $r_1$ of $u$ and $r_2$ of $\nabla u$ are allowed to be supercritical and imposes severe restrictions on the use of techniques due to the lack of embedding of Sobolev spaces into suitable spaces.

In the past years, singular elliptic equations have been challenging mathematicians. An extensive literature is devoted to such problems with singularity, specially from the point of view of theoretical analysis. The following problem has been studied in many papers
\begin{equation}\label{sing}
\left\{
\begin{array}{rcll}
-\Delta_p u & = &\eta(x)u^{-\alpha}  &\textrm{in }\ \Omega\\
u & > &0 &\textrm{in }\ \Omega\\
u & =& 0 &\textrm{on }\ \partial\Omega,
\end{array}
\right.
\end{equation}
where $\eta(x)\geq 0$ in $\Omega$ and $0<\alpha<1$. In \cite{Canino_Sciunzi_Trombetta}, Canino, Sciunzi and  Trombetta showed the existence and uniqueness of solutions to \eqref{sing} when $p\neq 2$ in many configurations. 

In \cite{Giacomoni}, Giacomoni, Schindler and Tak\'a\v{c} studied the existence and multiplicity of weak solutions to the $p$-Laplacian problem 
\begin{equation*} 
\left\{
\begin{array}{rcll}
-\Delta_p u & =& \lambda u^{-\delta} + u^s  &\textrm{in }\ \Omega\\
u & > &0 &\textrm{in }\ \Omega\\
u & =& 0 &\textrm{on }\ \partial\Omega,
\end{array}
\right.
\end{equation*}
where $1<p<\infty, p-1<$ $s \leq p^*-1, \lambda>0$, and $0<\delta<1$. As usual, $p^*=\frac{N p}{N-p}$ if $1<p<$ $N, p^* \in(p, \infty)$ is arbitrarily large if $p=N$, and $p^*=\infty$ if $p>N$.  

Also, singular problems driven by more general operators are also in the literature. For example, 
\begin{equation*} 
	\left\{
	\begin{array}{rcll}
-\diver\left(a(|\nabla u|^p)|\nabla u|^{p-2}\nabla u\right)&=&u^{-\alpha}+u^{\beta} &\textrm{in }\ \Omega\\
u & > &0 &\textrm{in }\ \Omega\\
u & =& 0 &\textrm{on }\ \partial\Omega,
\end{array}
\right.
\end{equation*}
was considered by Corrêa, Corrêa and Figueiredo in \cite{Giovany} for a bounded, smooth domain $\Omega\subset\mathbb{R}^N$ ($N\geq 3$), exponents $0<\alpha,\beta<p-1$ and a quite general function $a\colon\mathbb{R}^+\to\mathbb{R}^+$, which allows the authors to consider a wide range of problems, including the $p$-Laplacian operator $\Delta_p$ and the $(p,q)$-Laplacian among others.

More details on topics related to singular problems without convection terms can be found in \cite{BMZ2019,BPZ2022,Giovany, GN2012, PRW2023} and references therein. 

Elliptic problems with convection terms also have been considered in various frameworks. In \cite{Bueno-G}, using the method of sub- and super-solutions combined with a global estimate on the gradient, Bueno and Ercole prove the existence of at least one positive solution for the problem
\begin{equation*} 
\left\{
\begin{array}{rcll}
-\Delta_p u & =&  \beta f(x,u,\nabla u) +\lambda h(x,u) & \textrm{in }\ \Omega\\
u & =& 0 &\textrm{on }\ \partial\Omega,
\end{array}
\right.
\end{equation*}
where $h,f$ are continuous nonlinearities satisfying $0\leq \omega_1(x) u^{q-1} \leq h(x,u) \leq \omega_2(x) u^{q-1}$ with $1<q<p$ and $0 \leq f(x,u, \nabla u) \leq \omega_3(x) u^{a}|v|^b$ with $a,b>0$ and the functions $ \omega_i$, $1\leq i \leq 3$, are positive, continuous weights in $\bar{\Omega}$. It is important to note that the results presented in \cite{Bueno-G} are rare in the literature when the exponents of $u$ and $|\nabla u|$ are greater than $p-1$.

In \cite{fmm}, Faria, Miyagaki and Motreanu proved existence of a positive solution for the following quasi-linear elliptic problem involving the $(p,q)$-Laplacian and a convection term by aplying the Galerkin's method and a Schauder basis.
\begin{equation*} 
\left\{
\begin{array}{rcll}
-\Delta_p u -\mu\Delta_q u & = &  f(x,u,\nabla u)  & \textrm{in }\ \Omega\\
u&>& 0  &\textrm{in }\ \Omega\\
u & =& 0 &\textrm{on }\ \partial\Omega,
\end{array}
\right.
\end{equation*}

In \cite{Motreanu}, Liu, Motreanu, and Zeng, using the method of sub- and super-solution, truncation techniques, nonlinear regularity theory, Leray-Schauder alternative principle, and set-valued analysis, established a result of existence of a positive solution to the problem
\begin{equation*}
	\left\{
	\begin{array}{rcll}
		-\Delta_p u & = & f(x,u(x),\nabla u(x)) +g(x,u(x)) & \textrm{in }\ \Omega\\
		u & > &0 &\textrm{in }\ \Omega\\
		u & =& 0 &\textup{on }\ \partial\Omega,
	\end{array}
	\right.
\end{equation*}
where $f\colon \Omega \times \mathbb{R} \times \mathbb{R}^N \rightarrow \mathbb{R}$ satisfies a suitable growth condition and the semilinear function $g\colon \Omega \times (0, \infty) \rightarrow \mathbb{R}$ is singular at $s=0$, that is,
$\displaystyle\lim _{s \rightarrow 0^{+}} g(x, s)=+\infty$. It is noteworthy that the conditions \eqref{f} in this paper are significantly less restrictive than those on $f$ in \cite{Motreanu}.

We state the main results of the paper. First we consider the case $0<r_1,r_2<p-1$.
\begin{theorem}\label{main result2}
  Suppose $f$ is a continuous function satisfying \eqref{f} such that $r_1, r_2 < p-1$. Then the problem \eqref{ps} has at least one positive solution $u \in W^{1,p}_0(\Omega)$ for each $\lambda>0$.  
\end{theorem}\vspace*{.2cm}

In order to handle the case $p-1<r_1,r_2$ we need to assume a stronger assumption on the regularity of $\Omega$: either $\partial\Omega\in C^{1,1}$ or $\Omega$ convex and $\partial\Omega\in C^{1,\tau}$, $0<\tau<1$. In either case we are allowed to apply Propositions \ref{Prop1} and \ref{Prop2} (stated in Section \ref{preliminares}), which demand $\partial\Omega\in W^2 L^{N-1, 1}$, where $W^2 L^{N-1, 1}$ denotes the Lorentz-Sobolev space. As remarked by Cianchi and Maz'ya in \cite{Cianchi_Mazya}, the assumption $\partial \Omega \in W^2 L^{N-1,1}$ means that $\Omega$ is locally the subgraph of a function of $N-1$ variables whose second-order weak derivatives belong to the Lorentz space $L^{N-1,1}$ (see Section \ref{preliminares} for basic results about this space) and $\partial \Omega \in C^{1,0}$. 
	
So, with stronger hypotheses on $\partial\Omega$, we obtain the next result.
\begin{theorem}\label{main result1}
Let $\Omega \subset \mathbb{R}^N$ be a bounded domain such that $\partial \Omega \in C^{1,1}$ and $q>\max\{N, p'\}$, where $p'=\frac{p}{p-1}$. Suppose that $f$ is a continuous function satisfying \eqref{f}, $0<\alpha q<1$ and $r_1,r_2 > p-1$. Then there is a constant $A^*>0$ such that, for each $\lambda \in (0, A^*)$, there is $M:=M(\lambda)$ such that the problem \eqref{ps} has at least one positive solution $u_\lambda$ satisfying 
\begin{align*}
\lambda u_0 \leq u_\lambda \leq Mu_0\quad \text{ and } \quad \Vert \nabla u_\lambda \Vert_{\infty} \leq  M.
\end{align*}
\end{theorem}\vspace*{.2cm}

As a consequence of this result, we can obtain that, $u_{\lambda} \to 0$ and $|\nabla u_{\lambda}| \to 0$, uniformly in $\Omega$, as $\lambda \to 0^+$, see the last part of the proof of Theorem \ref{main result1}.

We now briefly describe the plan of the paper. We collect a few results on the $p$-Laplacian operator in Section \ref{preliminares}. In Section \ref{auxiliary problem} we define an auxiliary problem, which will be useful in the proof of Theorem \ref{main result2}. The existence of a positive solution to the auxiliary problem is obtained in Lemma \ref{teo aux} by using the Galerkin method and a Schauder basis in $W^{1,p}_0(\Omega)$ combined with a consequence of Brouwer's fixed point theorem. The proof of Theorem \ref{main result2} is then obtained by showing that the solution of the auxiliary problem converges to the solution of problem \eqref{ps}, using that $W^{1,p}_0(\Omega)$ is a reflexive space and the embedding $W^{1,p}_0(\Omega) \hookrightarrow L^p(\Omega)$. The comparison principle (see Lemma \ref{Cuesta_Takac}) and the solution $u_0$ of problem \eqref{P2} (see below) guarantee the positivity of the solution of problem \eqref{ps}.
	
The proof of Theorem \ref{main result1} is presented in Section \ref{teorema1} by applying Schauder's fixed point theorem. The main tools are Proposition \ref{lem3}, where we show the existence of suitable $\lambda$ and $M$ to use the comparison principle combined with the solution $u_0$ of problem \eqref{P2}. To show the $L^{\infty}$ estimates on the gradient we apply Proposition \ref{Prop2}, established in \cite{Grey}.

\section{Preliminaries} \label{preliminares}

Let $W_0^{1,p}(\Omega)$ stand for the usual Sobolev space endowed with the norm $\|u\|=\|\nabla u\|_p$. We denote $p^*=\frac{N p}{N-p}$ if $1<p<N$. In the case $p=N$,  then $p^* \in(p, \infty)$ can be taken arbitrarily large and $p^*=\infty$ if $p>N$. 

\begin{definition} We say that $u\in W_0^{1,p}(\Omega)$ is a solution (respectively, sub-solution and super-solution) of \eqref{ps} if $g_\lambda(\cdot,u,\nabla u):= \lambda u^{-\alpha} + f(x,u,\nabla u) \in L^p(\Omega)$,
\begin{equation*}
    \int_{\Omega} |\nabla u|^{p-2}\nabla u \cdot \nabla \varphi   dx - \int_\Omega g_\lambda(x,u,\nabla u)\varphi dx = 0\ \ (\leq 0, \geq 0) 
\end{equation*}
for all $\varphi\in W_0^{1,p}(\Omega)$, $\varphi\geq 0$ and 
\begin{equation*}
    u=0\ \ (\text{resp.}\ \leq 0, \geq 0) \mbox{ on } \partial\Omega.
\end{equation*}
The condition on $\partial\Omega$ is understood in $W^{1-1/p,p}(\partial\Omega)$, i.e, in the sense of traces.
\end{definition}

By an ordered pair of sub- and super-solutions we mean a sub-solution $\alpha$ and a super-solution $\beta$ such that $\alpha\leq \beta$ a.e. 

It is well-known that the equation $-\Delta_p u =v$ in $\Omega$ with Dirichlet boundary condition has a unique weak solution $u\in C^{1,\sigma}(\bar{\Omega})$ for some $\sigma \in (0,1)$ if $v \in L^{\infty}(\Omega)$, see \cite{Dib, Lieberman, Vazquez}. Moreover, if $v\geq 0$ and $v\not\equiv 0$, then $u>0$. Hence $\partial u / \partial \eta <0$ on $\partial\Omega$, where $\eta$ is the unit normal vector to $\partial \Omega$ pointing outward of $\Omega$. Furthermore,  $u$ is bounded from above and from below by positive multiples of the distance function $\dist(x,\partial\Omega)$. 

Additionally, the associated solution operator $(-\Delta_p)^{-1}: L^{\infty}(\Omega) \to C^{1}(\bar{\Omega})$ is positive, continuous and compact.  In synthesis, $(-\Delta_p)^{-1}$ can be viewed as  a strongly positive operator on $C(\bar{\Omega})$, i.e., $v \in P$ implies $(-\Delta_p)^{-1}v \in \mbox{int}(P)$, where $P$ denotes the cone of positive functions belonging $C(\bar{\Omega})$. 

The following comparison result is true in the case of a bounded domain $\Omega\subset\mathbb{R}^N$ and $u, v \in W^{1, p}(\Omega) \cap C(\bar{\Omega})$, $1 < p < \infty$. If $-\Delta_{p} u \leq-\Delta_{p} v$ on $\Omega$ in the weak sense and $u\leq v$ on $\partial \Omega$, then $u\leq v$ in $\Omega$. 

In order to consider problem \eqref{ps}, let us recall some results, starting with the Dirichlet problem
\begin{equation}\label{p-Laplacian_prob}
	\left\{\begin{array}{rcll}-\Delta_pu &=&f(x) &\text { in }\ \Omega \\ u&=&0 &\text { on }\ \partial \Omega.\end{array} \right. 
\end{equation}

The following comparison result, proved by Cuesta and Tak\'a\v{c} \cite[Proposition 2.3]{Cuesta_Takac}, will be very important in the sequel.  

\begin{lemma}\label{Cuesta_Takac}
Let $f$ be a non-negative continuous function such that $f(u)$ is nonincreasing for $u \in (0,\infty)$, where $1<p$. Assume that, for $p>1$, $u, v \in W^{1, p}(\Omega)$ are weak solutions such that
$$
\begin{array}{rclll}-\Delta_p u &\geq &f(u), & u>0 &\textrm {in } \Omega, \\ -\Delta_p v &\leq& f(v), &v>0 &\textrm{in } \Omega, \end{array}
$$
and $u\geq v$ in $\partial\Omega$. 

Then $u \geq v$ in $\Omega$.   
\end{lemma}

We now recall the definition of the Lorentz spaces. For a measurable function $\phi\colon \Omega \to \mathbb{R}$ and $t>0$, we denote by
$$\mu_{\phi}(t)=\big|\{x \in \Omega\,:\, |\phi(x)|>t \}\big|$$
its \textit{distribution function}, $|\cdot|$ standing for the Lebesgue measure in $\mathbb{R}^N$. 

For $0\leq r \leq|\Omega|$, the decreasing rearrangement $\phi^*$ of $\phi$ is defined by 
\[
\phi^*(r)=\sup\left\{t>0\,:\,  \mu_{\phi}(t)> r \right\},
\]
and the Lorentz space $L^{p,q}$ is defined for $1<p<\infty$ and $1\leq q<\infty$ by
\[
L^{p,q}=\left\{
\phi\colon\Omega \to \mathbb{R} \mbox{ measurable}\,:\,
\left(\int\limits_0^{|\Omega|}\frac{1}{t} [\phi^*(t)t^{1/p}]^q dt\right)^{1/q}<\infty\right\}
\]
endowed with norm
$$
\|\phi\|_{p,q}=\left(\int\limits_0^{|\Omega|}\frac{1}{t} [\phi^*(t)t^{1/p}]^q dt\right)^{1/q}.
$$

Denoting the usual norm in the space $L^q(\Omega)$ by $\|\cdot\|_q$, for $1<q<\infty$ we have the identity
\[
L^{q,q}=L^q(\Omega)\quad\textrm{ and }\quad \|\phi\|_{q,q} = \|\phi\|_{q}.\]
Furthermore, the following continuous embeddings are true for $1<s<N<q<\infty$.
\begin{equation*}
	L^q \hookrightarrow L^{N,1} \hookrightarrow L^{N,s} \hookrightarrow L^{N,N}=L^N \hookrightarrow L^{N,q} \hookrightarrow L^s,
\end{equation*}
see \cite{Adams} or  \cite{Tartar} for details.

The next two results are important because they relate the norm $\left\|\nabla u_p\right\|_{\infty}^{p-1}$ of a solution $u_p\in W_0^{1, p}(\Omega)$ of problem \eqref{p-Laplacian_prob} with the norm $\|f\|_q$ by means of an explicit dependence on $p$, although not necessarily optimal.
\begin{proposition}\textup{\cite[Theorem 1.3]{Grey}}\label{Prop1}
Let $\Omega$ be a bounded domain of $\mathbb{R}^2$ such that $\partial \Omega \in W^2 L^{\theta, 1}$ for some $\theta>1$. Suppose that $f \in$ $L^q(\Omega)$ for some $q>2$, and let $u_p\in W_0^{1, p}(\Omega)$ be a solution of problem \eqref{p-Laplacian_prob}. Then,
$$\left\|\nabla u_p\right\|_{\infty}^{p-1} \leq C \begin{cases}2^{\frac{p}{p-1}}(p-1)^{-\frac{2 \theta}{\theta-1}}\|f\|_q &\textrm {if }\ 1<p<2 \\ p^{\frac{5}{2}+\frac{2 \theta}{\theta-1}}\|f\|_q &\textrm{if }\ p \geq 2,\end{cases}
$$
for some constant $C$ depending at most on $q$ and $\Omega$. Moreover, if the assumption $\partial \Omega \in W^2 L^{\theta, 1}$ is replaced with the assumption that is $\Omega$ convex, then
$$
\left\|\nabla u_p\right\|_{\infty}^{p-1} \leq C \begin{cases}2^{\frac{p}{p-1}}\|f\|_q &\textrm {if }\ 1<p<2 \\ p^{\frac{5}{2}}\|f\|_q &\textrm{if }\ p \geq 2,\end{cases}
$$
for some constant $C$ depending at most on $q$ and $\Omega$.    
\end{proposition}

\begin{proposition}\textup{\cite[Theorem 1.2]{Grey}}\label{Prop2}
Let $\Omega$ be a bounded domain of $\mathbb{R}^N, N \geq 3$, such that $\partial \Omega \in W^2 L^{\theta, 1}$ for some $\theta>N-1$, and assume that $f \in L^{N, 1}(\Omega)$. Let $u_p \in W_0^{1, p}(\Omega)$ be a weak solution of the Dirichlet problem \eqref{p-Laplacian_prob}. Then the estimate
$$\left\|\nabla u_p\right\|_{\infty}^{p-1} \leq C \begin{cases}2^{\frac{p}{p-1}}(p-1)^{-\frac{\theta N}{\theta-(N-1)}}\|f\|_{N, 1} & \text { if } 1<p<2 \\ p^{\frac{5}{2}+\frac{\theta N}{\theta-(N-1)}}\|f\|_{N, 1} &\textrm {if }\ p \geq 2 .\end{cases}$$
holds for a constant $C$, depending at most on $N$ and $\Omega$. Moreover, if the assumption $\partial \Omega \in W^2 L^{\theta, 1}$ is replaced with the assumption that is $\Omega$ convex, then 
$$\left\|\nabla u_p\right\|_{\infty}^{p-1} \leq C \begin{cases}2^{\frac{p}{p-1}}\|f\|_{N, 1} &\textrm{if }\ 1<p<2 \\ p^{\frac{5}{2}}\|f\|_{N, 1}&\textrm{if }\ p \geq 2\end{cases}$$
holds for constant $C$ depending at most on $N$ and $\Omega$.    
\end{proposition}

The two last results were originally established by Chianchi and Maz'ya \cite{Cianchi_Mazya} in a more general context for general constants $C$. However, the explicit dependence of the constants on  $p$ was obtained by Ercole in \cite{Grey}.\vspace*{.2cm}

Now, suppose that there exist constants $0<c<\infty$ and $0<\delta<1$ such that
\begin{equation}\label{est_dist}
0 \leq f(x) \leq c\, \textup{d}_\Omega(x)^{-\delta} 
\end{equation}
holds for almost all $x \in \Omega$,
where $\textup{d}_{\Omega} \in W^{1, \infty}(\Omega) \cap C_0(\bar{\Omega})$ denotes the distance function to the boundary defined by
$$
\textup{d}_\Omega(x) := \dist(x,\partial\Omega) = \inf_{y\in\partial\Omega}|x-y|,\ x \in \bar{\Omega} .
$$

When $0<\delta<1$, the integrability requirement given by
\begin{equation*}
\int_{\Omega}d_\Omega(x)^{-\delta}dx <\infty, 
\end{equation*} 
is a consequence of \cite[Lemma 1]{Lazer} (see also \cite[Lemma B.1]{tese2}). Hence, the condition $0<\alpha q<1$ in Theorem \ref{main result1} implies that
\begin{equation}\label{integral_dist}
\int_{\Omega}d_\Omega(x)^{-\alpha q}dx <\infty.
\end{equation}

The following result is a simplified version of the one obtained by Giacomoni, Schindler and Tak\'a\v{c}. \cite[Theorem B.1]{Giacomoni}.
\begin{proposition}\label{giacomoni} Assume that $f$ satisfies the growth hypothesis \eqref{est_dist}. Let $u \in W_0^{1, p}(\Omega)$ be the (unique) weak solution of problem \eqref{p-Laplacian_prob}. In addition, assume
\begin{equation}\label{est_u_2}
0 \leq u(x) \leq C d_\Omega(x) \quad \text { for almost all } x \in \Omega,
\end{equation}
where $C$ is a constant, $0 \leq C<\infty$. Then there exist constants $\beta$ and $\Gamma, 0<\beta<1$ and $0 \leq \Gamma<\infty$, depending solely on $\Omega, N, p$, on the constants $c, \delta$ in \eqref{est_dist}, and on the constant $C$ in \eqref{est_u_2}, such that $u$ satisfies $u \in C^{1, \beta}(\bar{\Omega})$ and
\begin{equation*}
\|u\|_{C^{1, \beta}(\bar{\Omega})} \leq \Gamma.
\end{equation*}
\end{proposition}

To prove Theorems \ref{main result2}  and \ref{main result1} in this paper, we consider the unique solution  $u_0 \in W^{1,p}_0(\Omega) \cap C^{1, \beta}(\bar{\Omega})$ of the $p$-Laplacian problem
\begin{equation} \label{P2}
	\left\{
	\begin{array}{rcll}
		-\Delta_p u_0 & = &u_0^{ -\alpha}  &\textrm{in }\ \Omega\\
		u_0 & > &0 &\textrm{in }\ \Omega\\
		u_0 & = &0 &\textrm{on }\ \partial\Omega,
	\end{array}
	\right.
\end{equation}
see \cite[Theorem 1.3]{Canino_Sciunzi_Trombetta} for further results about this problem. 

The next result is technical and will be crucial in the proof of our Theorem \ref{main result1}.
\begin{proposition}\label{lem3} Let $u_0$ be the solution of \eqref{P2}. There are positive constants $M$ and $\lambda$ such that
\begin{enumerate}
\item[$(i)$] $aM^{r_1} \Vert u_0 \Vert^{r_1}_{\infty}+ b M^{r_2}\leq \dfrac{\lambda^{1-\alpha}}{\Vert u_0 \Vert^{\alpha}_{\infty}}$;
\item[$(ii)$] $2^{\frac{1}{p-1}} \lambda^{\frac{1-\alpha}{p-1}} \leq M$;
\item[$(iii)$] $M \leq \lambda^{\frac{2-p}{\alpha}}$. 
\end{enumerate}
\end{proposition}
\begin{proof} Of course, 
$$aM^{r_1}\Vert u_0 \Vert_{\infty}^{r_1} \leq \dfrac{\lambda^{1-\alpha}}{2\Vert u_0 \Vert_{\infty}^{\alpha}}\quad\Leftrightarrow\quad M \leq \dfrac{\lambda^{\frac{1-\alpha}{r_1}}}{\left(2a\Vert u_0 \Vert_{\infty}^{r_1+\alpha}\right)^{\frac{1}{r_1}}}$$ 
and 
$$b M^{r_2} \leq \dfrac{\lambda^{1-\alpha}}{2\Vert u_0 \Vert_{\infty}^{\alpha}}\quad\Leftrightarrow\quad M \leq \dfrac{\lambda^{\frac{1-\alpha}{r_2}}}{\left(2b\Vert u_0 \Vert_{\infty}^{\alpha}\right)^{\frac{1}{r_2}}}. $$

Comparing the inequality in $(ii)$ with the first expression obtained for $M$, we observe that
\[2^{\frac{1}{p-1}} \lambda^{\frac{1-\alpha}{p-1}} < \dfrac{\lambda^{\frac{1-\alpha}{r_1}}}{\left(2a\Vert u_0 \Vert_{\infty}^{r_1+\alpha}\right)^{\frac{1}{r_1}}}
\]
is equivalent to
\begin{equation}\label{lambda}
\lambda^{\frac{(1-\alpha)(r_1+1-p)}{r_1(p-1)}} < \dfrac{1}{2^{\frac{1}{p-1}}\left(2a\Vert u_0 \Vert_{\infty}^{r_1+\alpha}\right)^{\frac{1}{r_1}}}.
\end{equation}

Solving \eqref{lambda} for $\lambda$ we obtain
\[\lambda<\dfrac{1}{2^{\frac{r_1}{(1-\alpha)(r_1+1-p)}}\left(2a\Vert u_0 \Vert_{\infty}^{r_1+\alpha}\right)^{\frac{p-1}{(1-\alpha)(r_1+1-p)}}}.
\]
Similarly, we derive
\[
\lambda<\dfrac{1}{2^{\frac{r_2}{(1-\alpha)(r_2+1-p)}}\left(2b\Vert u_0 \Vert_{\infty}^{\alpha}\right)^{\frac{p-1}{(1-\alpha)(r_2+1-p)}}}.
\]
Now let us consider $(ii)$ and $(iii)$. We have 
\begin{equation*}
	2^{\frac{1}{p-1}} \lambda^{\frac{1-\alpha}{p-1}} < \lambda^{\frac{2-p}{\alpha}} \Leftrightarrow  \lambda^{\frac{\alpha(1-\alpha)+(p-1)(p-2)}{\alpha(p-1)}} < \dfrac{1}{2^{\frac{1}{p-1}}} \Leftrightarrow \lambda < \dfrac{1}{2^{\frac{\alpha}{\alpha(1-\alpha)+(p-1)(p-2)}}}.
\end{equation*}
Therefore, in order to obtain $(i)-(ii)-(iii)$ it is enough to define
\begin{align*}
A \coloneqq &\min\Bigg\{\dfrac{1}{2^{\frac{r_1}{(1-\alpha)(r_1+1-p)}}\left(2a\Vert u_0 \Vert_{\infty}^{r_1+\alpha}\right)^{\frac{p-1}{(1-\alpha)(r_1+1-p)}}}, \dfrac{1}{2^{\frac{r_2}{(1-\alpha)(r_2+1-p)}}\left(2b\Vert u_0 \Vert_{\infty}^{\alpha}\right)^{\frac{p-1}{(1-\alpha)(r_2+1-p)}}},\\ &\qquad\quad \dfrac{1}{2^{\frac{\alpha}{\alpha(1-\alpha)+(p-1)(p-2)}}}, 1 \Bigg\}
\end{align*}
and consider
\begin{equation*}
	0 < \lambda < A,
\end{equation*}
since the last inequality implies
\begin{equation*}
2^{\frac{1}{p-1}} \lambda^{\frac{1-\alpha}{p-1}} < \min\Bigg\{\dfrac{\lambda^{\frac{1-\alpha}{r_1}}}{\left(2a\Vert u_0 \Vert_{\infty}^{r_1+\alpha}\right)^{\frac{1}{r_1}}}, \dfrac{\lambda^{\frac{1-\alpha}{r_2}}}{\left(2b\Vert u_0 \Vert_{\infty}^{\alpha}\right)^{\frac{1}{r_2}}}, \lambda^{\frac{2-p}{\alpha}}\Bigg\}.
\end{equation*}

With such a choice of $\lambda$, we can choose $M$ satisfying 
\begin{equation*}
2^{\frac{1}{p-1}} \lambda^{\frac{1-\alpha}{p-1}} \leq M \leq \min\Bigg\{\dfrac{\lambda^{\frac{1-\alpha}{r_1}}}{\left(2a\Vert u_0 \Vert_{\infty}^{r_1+\alpha}\right)^{\frac{1}{r_1}}}, \dfrac{\lambda^{\frac{1-\alpha}{r_2}}}{\left(2b\Vert u_0 \Vert_{\infty}^{\alpha}\right)^{\frac{1}{r_2}}}, \lambda^{\frac{2-p}{\alpha}}\Bigg\},
\end{equation*} 
and we are done.
\end{proof}

We now turn our attention to the singular  problem
\begin{equation}\label{problem_Giac}
\left\{\begin{array}{rcll}-\Delta_p u&=&\frac{\lambda}{u^\delta}+u^s &\text{in }\ \Omega, \\
u&>&0 &\text{in }\ \Omega,\\
u&=&0,  &\text{on }\ \partial \Omega,\end{array}\right.
\end{equation}
where $1<p<\infty$, $p-1<s \leq p^*-1$, $\lambda>0$, and $0<\delta<1$. 

The following result is established in \cite[Lemma A.7]{Giacomoni}.
\begin{lemma}\label{lemma_Gia2}
Each positive weak solution $u$ of problem \eqref{problem_Giac} satisfies $$c_\lambda \textup{d}_\Omega \leq u \leq K_\lambda \textup{d}_\Omega\ \textrm{ a.e. in }\ \Omega,$$ where $0<c_\lambda \leq K_\lambda<\infty$ are some constants independent of $u$.
\end{lemma}

In the course of the proof of Theorems \ref{main result2} and \ref{main result1}, we denote by $\lambda_{1}$ and $\phi_{1}$ the first eigenpair of the $p$-Laplacian, that is,
\begin{equation} \label{u1}
	\left\{\begin{array}{rcll}
		-\Delta_{p} \phi_{1}  & =& \lambda_{1}\phi_{1}^{p-1} & \textup{in }\ \Omega\\
		\phi_{1} & =& 0 &\textup{on }\  \partial \Omega,
	\end{array}\right.
\end{equation}
with $\phi_{1}>0$ satisfying $\left\|\phi_{1}\right\|_{\infty}=1$. Note that we have $\phi_1>0$ in $\Omega$ and $\frac{\partial \phi_1}{\partial \nu}<0$ on $\partial \Omega$, see \cite[Theorem 5]{Vazquez}. Hence, since $\phi_1 \in C^1(\bar{\Omega})$, there are constants $\ell$ and $L, 0<\ell<L$, such that $\ell d_{\Omega}(x)\leq \phi_1(x) \leq L d_{\Omega}(x)$ for all $x \in \Omega$. 

\section{An auxiliary problem} \label{auxiliary problem}

For each $\varepsilon>0$, we consider the auxiliary problem
\begin{equation} \label{P_n}
\left\{\begin{array}{rcll}
		-\Delta_p u & =&  \frac{\lambda}{(|u|+ \varepsilon)^{\alpha}} + f(x,u, \nabla u) &  \textup{in }\ \Omega  \\
		u & > &0  &\textrm{in }\  \Omega   \\
		u & = &0  &\textrm{on }\ \partial\Omega .
	\end{array}
	\right.
\end{equation}

To prove Theorem \ref{main result2} we first show the existence of a solution to the auxiliary problem \eqref{P_n} by applying the Galerkin method. 

In order to do so, we state a useful consequence of Brouwer's fixed point theorem, which was proved in \cite{Araujo-F2}. The statement is a subtle generalization of the classical result, allowing us to work with an arbitrary norm $|\cdot|_m$ in $\R^m$. We denote by $(\cdot,\cdot)$ the usual inner product in $\R^m$.
\begin{lemma}\label{prop1}
Let $F\colon (\R^m, |\cdot|_m) \rightarrow (\R^m, |\cdot|_m)$ be a continuous function such that $\left( F(\xi),\xi\right)\geq 0$ for every $\xi \in \R^m$ with $|\xi|_{m}=R$ for some $R>0$. Then, there exists $z_0$ in the closed ball $\bar{B}^m_r(0)=\{z \in \R^m\,:\, |z|_{m}\leq R\}$ such that $F(z_0)=0$.
\end{lemma}

\begin{lemma}\label{teo aux}
For any $\varepsilon>0$ the problem \eqref{P_n} admits a positive weak solution $u_{\varepsilon}\in W^{1,p}_0(\Omega)\cap C^{1,\sigma}(\bar{\Omega})$ for each $\lambda >0$ and some $0<\sigma<1$.
\end{lemma}

\begin{proof} Fix $\varepsilon>0$ and consider $\mathcal{B}=\{e_1,...,e_m,...\}$ a Schauder basis of $W^{1,p}_0(\Omega)$. For each $m \in \mathbb{N}$ and $\xi=(\xi_1,...,\xi_m)\in \R^m$, define the $m$-dimensional subspace $V_m \coloneqq [e_1,...,e_m]\subset W^{1,p}_0(\Omega)$ and the norm
\begin{align*}
|\xi|_m \coloneqq \left \Vert \sum^m_{j=1} \xi_j e_j \right\Vert_{W^{1,p}_0(\Omega)}.
\end{align*}
It follows that the map $T_m\colon (\R^m, |\cdot|_m) \rightarrow \left(V_m, \Vert \cdot \Vert_{W^{1,p}_0(\Omega)}\right)$ given by $\displaystyle T_m(\xi)=\sum^m_{j=1} \xi_j e_j = u$ is an isometric isomorphism. Consider the function $F\colon\R^{m} \to \R^{m}$, $F(\xi)=(F_1(\xi),F_2(\xi),\dots, F_m(\xi))$,  given by
\[F_j(\xi)=\int_{\Omega} \vert \nabla u \vert^{p-2} \nabla u \nabla e_j-  \lambda\int_{\Omega} \dfrac{e_j}{(\varepsilon+|u|)^{\alpha}} -\int_{\Omega} f(x,u, \nabla u)e_j, \,\,\,1\leq j\leq m.\]

\noindent\textbf{Claim.} $F\colon\R^m\to \R^m$ satisfies the hypotheses of Lemma \ref{prop1}. 

We start showing that $F$ is continuous. Indeed, for $k\in\mathbb{N}$, suppose that $\zeta^k\rightarrow\zeta^0$ in $\R^m$ and consider $u_k=\sum_{i=1}^m\zeta^k_i e_i\in V_m$. For each fixed $j\in\{1,2,\dots,m\}$, note that
\[
F_j(\zeta^k)=\int_{\Omega} \vert \nabla u_k \vert^{p-2} \nabla u_k \nabla e_j-  \lambda\int_{\Omega} \dfrac{e_j}{(\varepsilon+|u_k|)^{\alpha}} -\int_{\Omega} f(x,u_k, \nabla u_k)e_j.
\]

We will only give attention to the term $\int_{\Omega} f(x,u_k, \nabla u_k)e_j$. Since 
\[
u_k=\sum_{i=1}^m\zeta^k_i e_i \to u_0\coloneqq \sum_{i=1}^m\zeta^0_i e_i \quad \mbox{ in } \quad W^{1,p}_0(\Omega),
\]
we obtain $u_k \to u_0$ and $\nabla u_k \to \nabla u_0$, both in $L^p(\Omega)$.

Since $r_1, r_2 < p-1$ it follows from \eqref{f} that
\[
 0 \leq f(x,t, \xi) \leq c_1 +  c_2|t|^{p-1}+c_3|\xi|^{p-1},
\]
in $\Omega\times \mathbb{R}\times \R^m$. The continuity of the Nemytskii operator yields
\begin{equation}\label{conv1}
f(\cdot, u_k,\nabla u_k) \to f(\cdot, u_0,\nabla u_0) \quad \mbox{ in } \quad L^{\frac{p}{p-1}}(\Omega).    
\end{equation}
Hence,
\begin{multline*}
	\left|\int_{\Omega} f(x,u_k, \nabla u_k)e_j -\int_{\Omega} f(x,u_0, \nabla u_0)e_j  \right|  
\end{multline*}
\begin{align}\label{conv2}
& \leq \|f(\cdot,u_k, \nabla u_k) - f(\cdot,u_0, \nabla u_0)\|_{L^{\frac{p}{p-1}}(\Omega)}\|e_j\|_{L^p(\Omega)}  \to 0  \mbox{ as }  n \to \infty.
\end{align}    
Thus, we conclude that $F_j(\zeta^k)\rightarrow F_j(\zeta^0)$, proving the continuity of $F$.

It follows from \eqref{f}, Holder's inequality and the pertinent Sobolev embedding that
$$\langle F(\xi), \xi \rangle \geq \Vert u \Vert^p_{W^{1,p}_0(\Omega)}-  \lambda c_1\Vert u \Vert^{1-\alpha}_{W^{1,p}_0(\Omega)} -c_2 \Vert u \Vert^{r_1+1}_{W^{1,p}_0(\Omega)}- c_3\Vert u \Vert^{r_2+1}_{W^{1,p}_0(\Omega)},
$$
where $c_1, c_2$ and $c_3$ are positive constants independent of $m$ and $u$. Because $r_1, r_2 < p-1$ and $1-\alpha<p$, we obtain that $\langle F(\xi), \xi \rangle$ is coercive, that is, for each $\lambda >0$ there exists $R=R(\lambda)>0$ large enough such that $|\xi|_m \coloneqq \Vert u \Vert_{W^{1,p}_0(\Omega)}=R$ implies $\langle F(\xi), \xi \rangle\geq 0$.

As a consequence of Lemma \ref{prop1}, there exists $z_0 \in \R^{m}$ with $|z_0|\leq R$ such that $F(z_0)=0$. Taking into account the isometric isomorphism $T_m: \R^m \to V_m$, there exist $u_m \in V_m$ with $\|u_m\|_{W^{1,p}_0 (\Omega)}\leq R$ such that, for all $v\in V_m$,
\begin{equation}\label{eq15}
	\int_{\Omega}|\nabla u_m|^{p-2}\nabla u_m\nabla v = \lambda\int_{\Omega} \dfrac{v}{(\varepsilon+|u_m|)^{\alpha}} + \int_{\Omega} f(x,u_m,\nabla u_m)v.
\end{equation}

Since $R$ does not depend on $m$, the sequence $(u_m)$ is bounded in $W^{1,p}_0(\Omega)$. Thus, for a subsequence, there exists $u \in W^{1,p}_0(\Omega)$ such that \begin{equation*}
	u_m \rightharpoonup u \,\,\, \mbox{weakly in} \,\,\, W^{1,p}_0(\Omega).
\end{equation*}

Consequently, we conclude that $u_m\to u$ in $L^p(\Omega)$ and $u_m(x)\to u(x)$ a.e. in $\Omega$.
Observe also that 
\begin{equation}\label{eq30}
	\|u\|_{W^{1,p}_0(\Omega)} \leq \liminf_{m \to \infty}\|u_m\|_{W^{1,p}_0(\Omega)} \leq R.
\end{equation}

We now claim that
\begin{equation}\label{eq16.2}
	u_m \to u \,\,\, \mbox{in} \,\,\, W^{1,p}_0(\Omega).
\end{equation}

Since $\mathcal{B}$ is a Schauder basis of $W^{1,p}_0(\Omega)$, there exists a unique sequence $(\alpha_n)_{n\geq 1}$ in $\mathbb{R}$  such that $u=\sum _{j=1}^{\infty}\alpha_j e_j $. Thus, as $m\to\infty$,
\begin{equation}\label{166}
	\phi_m = \sum_{j=1}^m \alpha _j e_j \rightarrow u \,\, \mbox{ in } W_0^{1,p}(\Omega).
\end{equation}

Considering the test function $(u_m-\phi_m)\in V_m$ in \eqref{eq15} yields
\begin{multline*}
	\int_{\Omega} \vert \nabla u_m \vert^{p-2} \nabla u_m \nabla (u_m - \phi_m) \hfill
\end{multline*}
\begin{align*}
& = \lambda\int_{\Omega} \dfrac{(u_m - \phi_m)}{(\varepsilon+|u_m|)^{\alpha}} + \int_{\Omega} f(x,u_m,\nabla u_m)(u_m - \phi_m)\\
&\leq \dfrac{\lambda}{\varepsilon^\alpha} \Vert u_m - \phi_m \Vert_{L^{1}(\Omega)}+ \tilde{c}_1\Vert u_m \Vert_{W^{1,p}_0(\Omega)}^{r_2}\Vert u_m - \phi_m \Vert_{L^{r_2+1}(\Omega)} + \\ & \quad+ \tilde{c}_2\Vert u_m \Vert_{W^{1,p}_0(\Omega)}^{r_1}\Vert u_m - \phi_m \Vert_{L^{r_1+1}(\Omega)} \rightarrow 0 \quad \text{as} \quad m \to \infty.
\end{align*}
where we use \eqref{f}, Holder's inequality, the boundness of $\Vert u_m \Vert_{W^{1,p}_0(\Omega)}$, \eqref{eq16.2} and \eqref{166}.
Thus, we obtain
\[
\lim_{m\rightarrow\infty} \displaystyle\int_{\Omega}|\nabla u_m|^{p-2}\nabla u_m\nabla (u_m-u)=0. 
\]
Now it is sufficient to apply the $(S_+)$ property of $-\Delta_p$ (see \cite[Proposition 3.5]{Montreanu-M-P}) to obtain \eqref{eq16.2}.

If $k \in \mathbb{N}$, then for every $m\geq k$ and $v_k\in V_k$, we obtain
\begin{align*} 
\int_{\Omega} \vert \nabla u_m \vert^{p-2} \nabla u_m \nabla v_k=  \lambda\int_{\Omega} \dfrac{v_k}{(\varepsilon+|u_m|)^{\alpha}} + \int_{\Omega} f(x,u_m\nabla u_m)v_k.
\end{align*}

Since $[V_k]_{k \in \mathbb{N}}$ is dense in $W^{1,p}_0(\Omega)$ we conclude that, for all $v \in W_0^{1,p}(\Omega)$,
\begin{align*} 
\int_{\Omega} \vert \nabla u_m \vert^{p-2} \nabla u_m \nabla v=  \lambda\int_{\Omega} \dfrac{v}{(\varepsilon+|u_m|)^{\alpha}}  + \int_{\Omega} f(x,u_m, \nabla u_m)v.
\end{align*}
Proceeding as in \eqref{conv1} and \eqref{conv2} we obtain
\begin{equation*}
f(\cdot, u_m,\nabla u_m) \to f(\cdot, u,\nabla u) \quad \mbox{ in } \quad L^{\frac{p}{p-1}}(\Omega)    
\end{equation*}
and
\begin{multline*}
\left|\int_{\Omega} f(x,u_m, \nabla u_m)v -\int_{\Omega} f(x,u, \nabla u)v  \right|
\end{multline*}
\begin{align*}
& \leq \|f(\cdot,u_m, \nabla u_m) - f(\cdot,u, \nabla u)\|_{L^{\frac{p}{p-1}}(\Omega)}\|v\|_{L^p(\Omega)} \to 0\ \mbox{ as }\ m \to \infty.
\end{align*}    
Therefore,
\begin{align*}
    \int_{\Omega}  f(x,u_m,\nabla u_m) v \rightarrow \int_{\Omega}  f(x,u,\nabla u) v, \quad \forall v \in W_0^{1,p}(\Omega).
\end{align*}

As before, we conclude that 
\begin{align*}
    \int_{\Omega} \vert \nabla u_m \vert^{p-2}\nabla u_m \nabla v \rightarrow \int_{\Omega} |\nabla u \vert^{p-2}\nabla u \nabla v,
\end{align*}
what yields, for all $v \in W_0^{1,p}(\Omega)$,
\begin{equation*} 
\int_{\Omega} \vert \nabla u \vert^{p-2} \nabla u \nabla v=  \lambda\int_{\Omega} \dfrac{v}{(\varepsilon+|u|)^{\alpha}}+ \int_{\Omega} f(x,u,\nabla u)v.
\end{equation*}

Furthermore, since $u^- \in W_0^{1,p}(\Omega)$, we have
\begin{align*}  
\int_{\Omega} \vert \nabla u \vert^{p-2} \nabla u \nabla u^- =  \lambda\int_{\Omega} \dfrac{u^-}{(\varepsilon+|u|)^{\alpha}}+  \int_{\Omega} f(x,u,\nabla u)u^-,
\end{align*}
from what follows
\begin{align*}
    -\Vert u^- \Vert_{W_0^{1,p}(\Omega)} \geq  \lambda\int_{\Omega/\{u(x)>0\}} \dfrac{u^-}{(\varepsilon+|u|)^{\alpha}}  \geq 0.
\end{align*}
Then $u_- \equiv 0$ a.e. in $\Omega$. Therefore,
\begin{equation*}
\int_{\Omega} \vert \nabla u \vert^{p-2} \nabla u \nabla v=  \lambda\int_{\Omega} \dfrac{v}{(\varepsilon+u)^{\alpha}}  + \int_{\Omega} f(x,u,\nabla u)v, \hspace{0.3cm}  \forall v \in W_0^{1,p}(\Omega).
\end{equation*}
The equation in \eqref{P_n} guarantees that $u\not\equiv0$. By \cite[Theorem 7.1]{Lady} we infer
that $u\in L^\infty(\Omega)$. Furthermore, regularity up
to the boundary ensures that $u\in C^{1,\sigma}(\bar{\Omega})$ with some
$\sigma\in(0,1)$, see \cite[Theorem 1]{Lieberman}. It results from the strong maximum principle that $u>0$ in
$\Omega$, completing the proof that $u$ is a
solution of problem \eqref{P_n}. We are done. \end{proof}

\section{Proof of Theorem \ref{main result2}} \label{teorema1}

\begin{proof} We consider first the existence of a solution by taking successively $\varepsilon=1/n$, $n\in \mathbb{N}$.  According to Lemma \ref{teo aux}, for each $\lambda >0$, there exists a solution $u_n \in C^{1,\sigma}(\bar{\Omega})$ of the auxiliary problem \eqref{P_n}, that is, for all $v \in W_0^{1,p}(\Omega)$ we have
\begin{align} \label{eqq3}
\int_{\Omega} \vert \nabla u_n \vert^{p-2} \nabla u_n \nabla v=  \lambda\int_{\Omega} \dfrac{v}{(\frac{1}{n}+u_n)^{\alpha}}  + \int_{\Omega} f(x,u_n, \nabla u_n)v,
\end{align}
and, in particular, \eqref{eqq3} is valid for all $v \in C^1_c(\Omega)$.

Taking into account \eqref{eq30} and proceeding as in the proof of \eqref{eq16.2}, we obtain a subsequence $(u_n)$ such that
\begin{equation}\label{eq16.2.2}
	u_n \to u \ \mbox{ in } \ W^{1,p}_0(\Omega).
\end{equation}

In order to handle the first integral on the right side of \eqref{eqq3}, let us consider $w= \lambda^{\frac{1}{p-1-\alpha}}u_0$, where $u_0$ is the unique solution of \eqref{P2}. Of course, $w$ is a \emph{positive} solution to the problem
\begin{equation}\label{P6} 
\left\{
\begin{array}{rcll}
-\Delta_p w & =& \frac{\lambda}{w^{\alpha}}  &\textrm{in }\ \Omega \\
w & =& 0  &\textrm{on }\ \partial\Omega.
\end{array}
\right.
\end{equation}
Since $u_n$ is a solution of \eqref{P_n}, we have
\begin{align*}
 -\Delta_p\left( u_n +\frac{1}{n}\right) &= -\Delta_p u_n \geq \frac{\lambda}{\left(u_n+ \frac{1}{n}\right)^{\alpha}}
\end{align*}
and $(u_n+1/n)>0$ on $\partial\Omega$. In particular, $u_n +\frac{1}{n}$ is a super-solution of \eqref{P2} and, as a consequence of Lemma \ref{Cuesta_Takac}, we obtain $u_n +\frac{1}{n} \geq w$. 

Let $\phi_1$ be the first eigenfunction of the $p$-Laplacian, as in \eqref{u1}. For each $\beta \in \left(0, \lambda_1^{-1/(p-1-\alpha)}\right)$ we have that $\beta \phi_1$ satisfies 
\begin{equation*} 
\left\{
\begin{array}{rcll}
-\Delta_p (\beta \phi_1) & \leq&\frac{1}{(\beta \phi_1)^{\alpha}}  &\textrm{in }\ \Omega \\
\beta \phi_1 & >& 0 &\textrm{in }\ \Omega \\
\beta \phi_1 & =& 0  &\textrm{on }\ \partial\Omega.
\end{array}
\right.
\end{equation*}
In particular, $\beta \phi_1$ is a sub-solution of \eqref{P2}. Another application of Lemma \ref{Cuesta_Takac} yields $u_0 \geq \beta \phi_1$. 

Since there is a constante $\ell >0$ such that $\ell d_{\Omega} \leq \phi_1$, we obtain 
\begin{equation}\label{u2}
    u_0 \geq k d_{\Omega},
\end{equation}
where $k= \ell \beta$. Therefore, 
\begin{equation*}
    w= \lambda^{\frac{1}{p-1-\alpha}}u_0 \geq k_1 d_{\Omega},
\end{equation*}
where $k_1= k\lambda^{\frac{1}{p-1-\alpha}}$. Consequently, $u_n +\frac{1}{n}\geq w \geq k_1d_{\Omega}$, what implies 
\begin{equation*} 
  \dfrac{|v|}{\left(u_n +\frac{1}{n}\right)^{\alpha}} \leq \dfrac{|v|}{\left(k_1d_{\Omega}\right)^{\alpha}}
\end{equation*}
for all $v \in C_0^1(\Omega)$.

It follows from \eqref{integral_dist} that  $\dfrac{v}{\left(k_1d_{\Omega}\right)^{\alpha}}$ is integrable.  Since $u_n + \frac{1}{n} \to u$ a.e in $\Omega$, we conclude that $u \geq k_1d_{\Omega}>0$ a.e. in $\Omega$ and 
\[
 \dfrac{v}{(\frac{1}{n}+u_n)^{\alpha}} \to  \dfrac{v}{u^{\alpha}}\ \textrm{ a.e. in }\ \Omega
\]
as $n\to\infty$.

Now, an application of Lebesgue's dominated convergence theorem  yields
\begin{equation} \label{convergencia}
\int_{\Omega} \dfrac{v}{(\frac{1}{n}+u_n)^{\alpha}} \to \int_{\Omega} \dfrac{v}{u^{\alpha}}\ \textup{ as }\  n \to +\infty
\end{equation}

Since $C^1_0(\Omega)$ is dense in $W^{1,p}_0(\Omega)$ and $u_n$ solves \eqref{eqq3}, it follows from \eqref{eq16.2.2} and \eqref{convergencia} as $n\to \infty$ that, for all $v \in W^{1,p}_0(\Omega)$,
\begin{equation*}
\int_{\Omega} |\nabla u |^{p-2} \nabla u \nabla v=  \lambda\int_{\Omega} \dfrac{v}{u^{\alpha}} + \int_{\Omega} f(x,u, \nabla u)v, \text{  } \forall v \in W^{1,p}_0(\Omega). 
\end{equation*} 
The proof is complete.
\end{proof}

\section{Proof of Theorem \ref{main result1}}\label{teorema2}
\begin{proof} Let us consider $\lambda$ and $M$ satisfying Proposition \ref{lem3}. Define the set
\begin{equation*}
\mathcal{A}:=\left\{v \in W^{1,p}_0(\Omega): \lambda u_0 \leq v \leq Mu_0,\text{ } \Vert \nabla v \Vert_{\infty} \leq M \right\} .
\end{equation*}

We will prove that for each $v \in \mathcal{A}$ there exists an unique positive solution $u$ of the problem
\begin{equation} \label{aux.problem}
\left\{
\begin{array}{rcll}
-\Delta_{p} u & =&  \lambda v^{-\alpha} + f(x,v, \nabla v)&\textrm{in }\ \Omega \\
u & >&0 &\textrm{in }\ \Omega \\
u & =& 0 &\textrm{on }\ \partial \Omega .
\end{array}
\right.
\end{equation}

We know that there are positive constants $\ell$ and $C_2$ such that 
\[\ell d_\Omega(x) \leq \phi_1(x) \leq C_2u_0(x),
\]
thus yielding, for $c=\ell/C_2$,
\begin{align} \label{d1}
cd_\Omega(x) \leq u_0(x).
\end{align} 

If $p'=\frac{p}{p-1}$ and $q> \max\{N,p'\}$, by our hypothesis we have $0<\alpha q< 1$. Then
\begin{align*}
    \dfrac{1}{u_0(x)^{q\alpha}} \leq \dfrac{1}{[cd_\Omega(x)]^{q\alpha}}.
\end{align*}
Thus, it follows from \eqref{integral_dist} that
\begin{align*}
    \int_{\Omega} \left(\dfrac{1}{u_0(x)^{\alpha}}\right)^q = \int_{\Omega} \dfrac{1}{u_0(x)^{q\alpha}} \leq \int_{\Omega} \dfrac{1}{[cd_\Omega(x)]^{q\alpha}} < \infty,
\end{align*}
and, since $\lambda u_0 \leq v$,
\begin{align} \label{equ}
    \left( \int_{\Omega} \left| \dfrac{\lambda}{v^{\alpha}}+ f(x,v,\nabla v)\right|^q\right)^{\frac{1}{q}} 
    &\leq \left( \int_{\Omega} \left( \left| \dfrac{\lambda}{v^{\alpha}} \right|+ a|v|^{r_1}+ b|\nabla v|^{r_2}  \right)^q\right)^{\frac{1}{q}} \nonumber\\
    &\leq \left( \int_{\Omega} \left( \dfrac{\lambda^{1-\alpha}}{|u_0|^{\alpha}}+ aM^{r_1}\Vert u_0\Vert_{\infty}^{r_1}+ b M^{r_2} \right)^q\right)^{\frac{1}{q}} \nonumber\\
     &\leq \left( \int_{\Omega} \left( \dfrac{\lambda^{1-\alpha}}{|u_0|^{\alpha}}+\dfrac{\lambda^{1-\alpha}}{|u_0|^{\alpha}} \right)^q\right)^{\frac{1}{q}}= \left( \int_{\Omega} \left( \dfrac{2\lambda^{1-\alpha}}{|u_0|^{\alpha}}\right)^q\right)^{\frac{1}{q}}\nonumber\\
     &= 2\lambda^{1-\alpha} \left\Vert \dfrac{1}{u_0^{\alpha}}\right\Vert_{L^q(\Omega)} < \infty.
\end{align}

We conclude that $\dfrac{\lambda}{v^{\alpha}}+ f(x,v,\nabla v) \in L^q(\Omega)$ and, since $q>p'$, it follows that $$\dfrac{\lambda}{v^{\alpha}}+ f(x,v,\nabla v) \in L^{p'}(\Omega).$$ 
By applying \cite[Theorem 8]{Dinca}, we obtain the existence and uniqueness of the solution $u$. 

We now show that $u\in \mathcal{A}$, proving first that $\lambda u_0\leq u\leq Mu_0$. In fact, let $v \in \mathcal{A}$. Proposition \ref{lem3} guarantees that
\begin{align*}
    -\Delta_p u\geq \dfrac{\lambda}{v^{\alpha}} \geq \dfrac{\lambda}{M^{\alpha}u_0^{\alpha}} = \dfrac{\lambda}{M^{\alpha}}\dfrac{1}{u_0^{\alpha}} \geq \lambda^{p-1} \dfrac{1}{u_0^{\alpha}} = \lambda^{p-1} (-\Delta_p u_0)= -\Delta_p (\lambda u_0).
\end{align*}
The comparison principle imply that $u \geq \lambda u_0$.

On the other hand, it follows from Proposition \ref{lem3} and \eqref{f} that 
\begin{align*}
-\Delta_p u &=\dfrac{\lambda}{v^{\alpha}} + f(x,v,\nabla v) \leq \dfrac{\lambda}{\lambda^{\alpha}u_0^{\alpha}} + a|v|^{r_1}+b |\nabla v|^{r_2}\\
& \leq \dfrac{\lambda^{1-\alpha}}{u_0^{\alpha}} + aM^{r_1}\Vert u_0 \Vert_{\infty} + b M^{r_2} \leq \dfrac{\lambda^{1-\alpha}}{u_0^{\alpha}} + \dfrac{\lambda^{1-\alpha}}{\Vert u_0 \Vert_{\infty}^{\alpha}} \\
&\leq \dfrac{2\lambda^{1-\alpha}}{u_0^{\alpha}} = 2 \lambda^{1-\alpha} (-\Delta_p u_0) = -\Delta_p (2^{\frac{1}{p-1}}\lambda^{\frac{1-\alpha}{p-1}} u_0).
\end{align*}
The comparison principle and Proposition \ref{lem3} imply that $u \leq 2^{\frac{1}{p-1}}\lambda^{\frac{1-\alpha}{p-1}} u_0 \leq M u_0$, that is, 
\begin{equation}\label{cond_invariance}
\lambda u_0 \leq u \leq M u_0.
\end{equation}

We will now show $\Vert \nabla u \Vert_{\infty} \leq  M$. If $N=2$, since $\dfrac{\lambda}{v^{\alpha}}+ f(x,v,\nabla v) \in L^q(\Omega)$, it follows from Proposition \ref{Prop1} that
\begin{align*}
    \Vert \nabla u \Vert^{p-1}_{\infty} \leq c_p\left\Vert \dfrac{\lambda}{v^{\alpha}}+ f(\textbf{ . },v,\nabla v) \right\Vert_{L^q(\Omega)} \leq c_p 2\lambda^{1-\alpha} \left\Vert \dfrac{1}{u_0^{\alpha}}\right\Vert_{L^q(\Omega)},
\end{align*}
and therefore
\[
\Vert \nabla u \Vert_{\infty} \leq \tilde{c_p} 2^{\frac{1}{p-1}}\lambda^{\frac{1-\alpha}{p-1}},
\]
where $\tilde{c_p}= c_p^{\frac{1}{p-1}}$.

If $N \geq 3$, since $q>N$, we have $L^q(\Omega) \subset L^{N,1}(\Omega)$, allowing us to conclude that $\dfrac{\lambda}{v^{\alpha}}+ f(x,v,\nabla v) \in L^{N,1}(\Omega)$. So, \eqref{equ} yields 
\begin{align*}
\left\Vert \dfrac{\lambda}{v^{\alpha}}+ f(\textbf{ . },v,\nabla v) \right\Vert_{L^{N,1}(\Omega)} \leq \overline{c} \left\Vert \dfrac{\lambda}{v^{\alpha}}+ f(\textbf{ . },v,\nabla v) \right\Vert_{L^q(\Omega)} <\infty.
\end{align*}

It follows from Proposition \ref{Prop2} that  
\begin{align*}
    \Vert \nabla u \Vert^{p-1}_{\infty} &\leq C_p\left\Vert \dfrac{\lambda}{v^{\alpha}}+ f(\textbf{ . },v,\nabla v) \right\Vert_{L^{N,1}(\Omega)} \\ &\leq C_p\, \bar{c} \left\Vert \dfrac{\lambda}{v^{\alpha}}+ f(\textbf{ . },v,\nabla v) \right\Vert_{L^q(\Omega)} &\leq C_p\, \bar{c}\, 2\lambda^{1-\alpha} \left\Vert \dfrac{1}{u_0^{\alpha}}\right\Vert_{L^q(\Omega)}
    \end{align*}
and we conclude that
$$\Vert \nabla u \Vert_{\infty} \leq \hat{C_p} 2^{\frac{1}{p-1}}\lambda^{\frac{1-\alpha}{p-1}},$$
where $\hat{C_p}= \left(C_p\overline{c}\right)^{\frac{1}{p-1}}$. 

Thus, by considering $\Tilde{C_p}=\max \{\Tilde{c_p}, \hat{C_p}\}$, we have for all $N\geq 2$ 
\begin{align*}
    \Vert \nabla u \Vert_{\infty} \leq \Tilde{C_p} 2^{\frac{1}{p-1}}\lambda^{\frac{1-\alpha}{p-1}}.
\end{align*}

Now, choosing $\lambda$ so that $0< \lambda< A^*$, where
\begin{align*}
A^* \coloneqq \min\left\{\dfrac{1}{2^{\frac{r_1}{(1-\alpha)(r_1+1-p)}}\left(2a\Vert u_0 \Vert_{\infty}^{r_1+\alpha}\right)^{\frac{p-1}{(1-\alpha)(r_1+1-p)}}},\dfrac{1}{2^{\frac{r_2}{(1-\alpha)(r_2+1-p)}}\left(2b\Vert u_0 \Vert_{\infty}^{\alpha}\right)^{\frac{p-1}{(1-\alpha)(r_2+1-p)}}},  \right. \\
 \dfrac{1}{2^{\frac{\alpha}{\alpha(1-\alpha)+(p-1)(p-2)}}}, \dfrac{1}{\left(\tilde{C_p}2^{\frac{1}{p-1}}\right)^{\frac{r_1(p-1)}{(1-\alpha)(r_1+1-p)}}\left(2a\Vert u_0 \Vert_{\infty}^{r_1+\alpha}\right)^{\frac{p-1}{(1-\alpha)(r_1+1-p)}}},\\
\left. \dfrac{1}{\left(\tilde{C_p}2^{\frac{1}{p-1}}\right)^{\frac{r_2(p-1)}{(1-\alpha)(r_2+1-p)}}\left(2b\Vert u_0 \Vert_{\infty}^{\alpha}\right)^{\frac{p-1}{(1-\alpha)(r_2+1-p)}}}, \dfrac{1}{\left(\tilde{C_p}2^{\frac{1}{p-1}}\right)^{\frac{\alpha(p-1)}{\alpha(1-\alpha)+(p-1)(p-2)}}}, 1 \right\},
\end{align*}
we obtain $M=M(\lambda)$ such that
\[
\tilde{C_p}2^{\frac{1}{p-1}}\lambda^{\frac{1-\alpha}{p-1}}\leq M \leq \min\Bigg\{\dfrac{\lambda^{\frac{1-\alpha}{r_1}}}{\left(2a\Vert u_0 \Vert_{\infty}^{r_1+\alpha}\right)^{\frac{1}{r_1}}}, \dfrac{\lambda^{\frac{1-\alpha}{r_2}}}{\left(2b\Vert u_0 \Vert_{\infty}^{\alpha}\right)^{\frac{1}{r_2}}}, \lambda^{\frac{2-p}{\alpha}}\Bigg\}. 
\]
Thus
\begin{align*}
\Vert \nabla u \Vert_{\infty} \leq \tilde{C_p} 2^{\frac{1}{p-1}}\lambda^{\frac{1-\alpha}{p-1}} \leq M.
\end{align*}

\noindent\textbf{Claim.} We have $u \in C^{1, \beta}(\bar{\Omega})$ for some $0<\beta<1$, uniformly with respect to $v \in \mathcal{A}$.\vspace*{.2cm}

In order to prove the Claim, we apply Proposition \ref{giacomoni}. Indeed, since each solution $U$ of \eqref{problem_Giac} is a super-solution of \eqref{P6}, the comparison principle (Lemma \ref{Cuesta_Takac}) yields $w=\lambda^{\frac{1}{p-1+\alpha}} u_0\leq U$.

By Lemma \ref{lemma_Gia2} there is a positive constant $K_{\lambda}$ such that $U\leq K_{\lambda}d_\Omega$, consequently, $u_0 \leq K_{\lambda}d_\Omega$. Moreover, by \eqref{u2} there is a positive constante $k$ such that $u_0 \geq k d_\Omega$.  Furthermore, since 
\begin{align*} 
\dfrac{\lambda}{v^{\alpha}}+ f(x,v,\nabla v) &\leq \dfrac{\lambda}{v^{\alpha}} + a|v|^{r_1}+ b|\nabla v|^{r_2}  \\
&\leq  \dfrac{\lambda^{1-\alpha}}{|u_0|^{\alpha}}+ aM^{r_1}\Vert u_0\Vert_{\infty}^{r_1}+ b M^{r_2} \leq\dfrac{2\lambda^{1-\alpha}}{|u_0|^{\alpha}}\\
&\leq \dfrac{K_1}{d_\Omega^{\alpha}} = K_1d_\Omega^{-\alpha} \text{ a.e in } \Omega, 
\end{align*}
where $K_1=\frac{2\lambda^{1-\alpha}}{k^{\alpha}}$, as a consequence of \eqref{d1}. Moreover, \eqref{cond_invariance} implies that 
\begin{align*}
    u \leq M u_0 \leq MK_{\lambda}d_\Omega.
\end{align*}
So, Proposition \ref{giacomoni} guarantees that the solution of \eqref{aux.problem} satisfies $u \in C^{1, \beta}(\bar{\Omega})$ for some $0<\beta<1$.

As consequence of the previous arguments, the operator
$$
\begin{array}{rll}
T: \mathcal{A} & \longrightarrow & W^{1, p}_0(\Omega)  \\
v & \longmapsto & u,
\end{array}
$$
is well-defined, $u$ being the unique positive solution of \eqref{aux.problem}. Moreover, $T$ is continuous and compact. In fact, let $(v_n)_{n \in \mathbb{N}} \subset \mathcal{A}$ be a bounded sequence in $W^{1,p}_0(\Omega)$ such that $T(v_n)=u_n.$ Since 
    \begin{align*} 
    & \dfrac{\lambda}{v_n^{\alpha}}+ f(x,v_n,\nabla v_n) \leq \dfrac{K_1}{d_\Omega^{\alpha}} = K_1d_\Omega^{-\alpha}, \text{ for all } n \in \mathbb{N},
\end{align*}
arguments like the previous ones and Proposition \ref{giacomoni} yield that $$\Vert u_n \Vert_{ C^{1,\beta}(\bar{\Omega})} \leq \Gamma$$ or, equivalently,
\begin{equation*}
 \Vert T(v_n) \Vert_{ C^{1,\beta}(\bar{\Omega})} \leq \Gamma. 
\end{equation*}
Hence, $\mathcal{A}$ is a equicontinuous subset. Therefore, there is a convergent subsequence of $(T(v_n))_{n \in \mathbb{N}}$ in $C^1(\bar{\Omega})$ which, in particular, is convergent in $W^{1,p}_0(\Omega)$. 
Thus, since $\mathcal{A}$ is a bounded, convex set invariant under $T$, it follows from Schauder's Fixed Point Theorem the existence of a fixed point $u_\lambda$ for $T$. Of course, the fixed point $u_lambda$ satisfies \eqref{ps}, since $-\Delta_{p} u=\lambda u^{-\alpha} + f(x,u, \nabla u)$. 
\end{proof}

\subsection*{Availability of data and materials}
This declaration is not applicable.
\subsection*{Competing interests:}
The authors declare that they have no conflict of interests.
\subsection*{Authors' contributions}
All authors contributed equally to the article.

\end{document}